\documentclass[12pt]{amsart}

\usepackage{amssymb, amsmath, amsthm}
\usepackage[margin=1in]{geometry}

\allowdisplaybreaks

\numberwithin{equation}{section}

\theoremstyle{plain}

\newtheorem{thmx}{Theorem}
\newtheorem{lmmx}[thmx]{Lemma}
\theoremstyle{definition}
\newtheorem{remx}[thmx]{Remark}
\theoremstyle{definition}

\theoremstyle{remark}
\newcommand{\ichi}{\mathbf{1}}
\newcommand{\C}{\mathbb{C}}

\newcommand{\N}{\mathbb{N}}

\newcommand{\R}{\mathbb{R}}
\newcommand{\T}{\mathbb{T}}
\newcommand{\Z}{\mathbb{Z}}

\newcommand{\calF}{\mathcal{F}}

\newcommand{\calS}{\mathcal{S}}

\newcommand{\pqr}{ L^{p_1} \times L^{p_2} \to L^{p} }
\newcommand{\pqramalgam}{ (L^{p_1}, \ell^{q_1}) \times 
(L^{p_2}, \ell^{q_2}) \to (L^{p}, \ell^{q}) }
\newcommand{\supp}{{\,\mathrm{supp}\,}}
\newcommand{\diam}{{\,\mathrm{diam}\,}}

\newcommand{\period}{{\,\mathrm{period}\,}}

\begin{document}
\title{Notes on bilinear lattice bump Fourier multipliers} 

\author[Kato]{Tomoya Kato}
\author[Miyachi]{Akihiko Miyachi}
\author[Tomita]{Naohito Tomita}

\address[T. Kato]
{Gunma University, Kiryu, Gunma 376-8515, Japan.}

\address[A. Miyachi]
{Tokyo Woman's Christian University, 
Tokyo 167-8585, Japan.}

\address[N. Tomita]
{Osaka University, Osaka 560-0043, Japan.}

\email[T. Kato]{t.katou@gunma-u.ac.jp}
\email[A. Miyachi]{miyachi@lab.twcu.ac.jp}
\email[N. Tomita]{tomita@math.sci.osaka-u.ac.jp}

\date{\today}


\begin{abstract}
We consider the bilinear Fourier multiplier operator
with the multiplier written as a linear combination
of a fixed bump function.
For those operators
we prove two transference theorems,
one in amalgam spaces
and the other in Wiener amalgam spaces.
\end{abstract}

\maketitle


\section{Introduction}

For $\sigma \in L^{\infty}(\R^n \times \R^n)$, 
the {\it bilinear Fourier multiplier operator} 
$T_{\sigma}$ is defined by 
\begin{align*}
&
T_{\sigma} (f_1,f_2) (x) 
=
\iint
_{\R^n \times \R^n} 
 e^{ 2\pi i x\cdot (\xi_1 + \xi_2)}
\sigma (\xi_1, \xi_2) 
\widehat{f_1}(\xi_1) 
\widehat{f_2}(\xi_2)\, 
d\xi_1 d\xi_2, 
\\
& 
x\in \R^n, 
\quad 
f_1,f_2 \in \calS (\R^n), 
\end{align*}
where 
$\widehat{f_1}, \widehat{f_2}$ denote 
the Fourier transforms. 

Let $X_1,X_2$, and $Y$ be function spaces on $\R^n$ 
equipped with quasi-norms 
$\|\cdot \|_{X_1}$, $\|\cdot \|_{X_2}$, 
and $\|\cdot \|_{Y}$, 
respectively.  
If there exists a constant $C\in [0, \infty)$ such that 
\begin{equation*}
\|T_{\sigma}(f_1,f_2)\|_{Y}
\le C \|f_1\|_{X_1} \|f_2\|_{X_2}, 
\quad 
f_1\in \calS \cap X_1, 
\quad  
f_2\in \calS \cap X_2, 
\end{equation*}
then we denote 
the smallest possible $C$ by 
$\|T_{\sigma}\|_{X_1 \times X_2 \to Y}$.  
If there exists no such finite constant $C$, 
then we define 
$\|T_{\sigma}\|_{X_1 \times X_2 \to Y}=\infty$. 
We shall simply call $\|T_{\sigma}\|_{X_1 \times X_2 \to Y}$  
the operator norm of $T_{\sigma}$ in $X_1 \times X_2 \to Y$.

The bilinear Fourier multiplier operator 
was introduced by 
Coifman--Meyer \cite{CM-1975, 
CM-AIF, CM-Ast, CM-Beijing}  
and there have been many works. 
In the present article, 
we shall be interested in 
the multiplier of the following special form. 
For $a\in \ell^{\infty}(\Z^n \times \Z^n)$ and 
$\Phi\in C_{0}^{\infty}(\R^n \times \R^n)$, 
we define 
\[
\sigma_{a, \Phi}(\xi_1, \xi_2)=\sum_{\mu_1, \mu_2 \in \Z^n} 
a(\mu_1, \mu_2) \Phi (\xi_1 - \mu_1, \xi_2 - \mu_2), 
\quad \xi_1, \xi_2 \in \R^n. 
\]
For notational convenience, we write the corresponding bilinear 
operator as 
\[
T_{a, \Phi} = T_{\sigma_{a,\Phi}} . 
\]

The multiplier $\sigma_{a, \Phi}$ can be considered as a test case 
of more general bilinear Fourier multipliers.  
It was considered in some form or other in several papers. 
In the papers \cite{GHH, GHS}, 
the authors used the estimates for operators of the form 
$T_{a, \Phi}$ as key tools to prove boundedness of 
bilinear singular integrals with rough kernels.   
A study wholly focusing on $\sigma_{a, \Phi}$ was given recently 
by Bri\'ankova--Garafakos--He--Honz\'ik \cite{BGHH}, 
where the authors call $\sigma_{a, \Phi}$ the 
{\it lattice bump multiplier\/}. 
The main result of \cite{BGHH} gives estimate 
for the operator norm of  
$T_{a, \Phi}$ in $L^{p_1} \times L^{p_2} \to L^p$, 
$1/p=1/p_1+1/p_2$, in terms of 
$\|a\|_{\ell^{\infty}}$ and the cardinality of 
$\supp a$ (see Theorem 1.2 and Remark 1.1 of \cite{BGHH}), 
which generalize the estimates given in 
\cite{GHH, GHS}. 
In \cite{KMT-arxiv, KMT-arxiv-2, KMT-arxiv-3},  
the present authors 
considered bilinear Fourier multipliers $\sigma$ 
satisfying the estimates 
\[
\big| 
\partial_{\xi_1}^{\alpha}
\partial_{\xi_2}^{\beta}
\sigma (\xi_1, \xi_2) 
\big| 
\le C_{\alpha, \beta} W(\xi_1, \xi_2)
\]
with a fixed nonnegative function $W$, 
and gave some sufficient conditions on $W$ for 
$T_{\sigma}$ to be bounded in 
$L^2$-based amalgam spaces and in Wiener amalgam spaces. 
The results of \cite{KMT-arxiv, KMT-arxiv-2, KMT-arxiv-3} 
imply the estimates for the operator norm of $T_{a, \Phi}$ 
in terms of the absolute values $|a(\mu_1, \mu_2)|$, 
which cover the estimate of \cite{BGHH}.

In the present article, we shall not consider any particular 
estimates of the operator norm of $T_{a, \Phi}$ 
but we shall consider 
some {\it transference theorems\/} 
for $T_{a, \Phi}$. 
The transference theorem 
was first given by de Leeuw \cite{deLeeuw}, 
who proved that,   
under certain condition on the multiplier $m(\xi)$ on $\R$,  
if the Fourier multiplier operator   
$T_{m}$ is bounded in $L^{p}(\R)$, $p\in [1, \infty]$,  
then the periodic Fourier multiplier operators 
$T^{\period}_{m(\epsilon \cdot )}$, $\epsilon \in (0, \infty)$,  
are uniformly bounded in $L^{p}(\T)$, 
where $T_{m}$ and 
$T^{\period}_{m(\epsilon \cdot ) }$ are defined by 
\[
T_{m}f (x) = 
\int_{\R} 
e^{2\pi i x \xi} 
m(\xi) 
\widehat{f}(\xi)\, d\xi, 
\quad 
x \in \R, 
\quad 
f\in \calS (\R), 
\] 
and 
\[ 
T^{\period}_{m(\epsilon \cdot )} 
F (x) = 
\sum_{\mu \in \Z} 
e^{2\pi i x \mu} 
m(\epsilon \mu) 
\widehat{F}(\mu), 
\quad 
x \in \T=\R / \Z, 
\quad 
F \in C^{\infty} (\T) 
\] 
($\widehat{F}(\mu)$ denotes the Fourier coefficient of 
$F$). 
The converse to this theorem was given by 
Igari \cite[Theorem 2]{Igari 1969} and 
Stein--Weiss \cite[Theorems 3.18 
in Chapter VII]{stein-weiss}.  
Transference theorems were also given 
in several different settings; 
see  
\cite{Kenig-Tomas, Kaneko, Torres, Auscher-Carro, 
Fan, Kaneko-ESato}.   
Transference theorems for bilinear Fourier multipliers were 
given by Fan--Sato \cite{Fan-SSato}.

The purpose of the present article is to 
give two transference theorems 
for the bilinear operators $T_{a, \Phi}$.  
With $a\in \ell^{\infty}(\Z^n \times \Z^n)$, 
we shall associate two other operators. 
One is the bilinear Fourier multiplier operator $T^{\period}_{a}$ 
that acts on periodic functions and 
the other is the  bilinear operator $S_{a}$ that 
acts on sequence spaces. 
Under certain conditions on $\Phi$, 
we shall prove that 
$T_{a, \Phi}$ is bounded in amalgam spaces 
if and only if $T^{\period}_{a}$ is bounded in 
corresponding $L^p$ spaces, 
and 
$T_{a, \Phi}$ is bounded in Wiener amalgam spaces 
if and only if $S_{a}$ is bounded in 
corresponding $\ell^q$ spaces. 
Precise statements 
will be given in Theorems  
\ref{th-amalgam-main} and 
\ref{th-Wiener-amalgam-main}.

Most of the techniques used 
in the present article are in fact 
well-known in the 
theory of transference theorems. 
More directly, our arguments are 
modifications of those given in 
\cite{KMT-arxiv, KMT-arxiv-2, KMT-arxiv-3}.

Throughout this article, 
we use the following notations:   
$\langle z \rangle = (1+|z|^2)^{1/2}$ for $z\in \R^n$;  
$Q=(-1/2, 1/2]^{n}$ is the unit cube centered at the origin; 
$K Q=(-K/2, K/2]^n$ for $K\in (0, \infty)$; 
the Fourier transform of $f\in \calS (\R^n)$ 
is denoted by $\widehat{f}$ or by 
$\calF f$; 
the inverse Fourier transform is denoted by 
$\calF^{-1}$; 
for $m\in L^{\infty}(\R^n)$, 
the linear Fourier multiplier operator is defined by 
\[
m(D) f (x) 
=
\int
_{\R^n} 
 e^{ 2\pi i x\cdot \xi}
m (\xi) 
\widehat{f}(\xi)\, 
d\xi, 
\quad 
x\in \R^n, 
\quad 
f \in \calS (\R^n). 
\]

\section{The case of amalgam spaces}
\label{amalgam}

In this section, we shall give our 
first main theorem, which concerns 
the operator norm of $T_{a, \Phi}$ in amalgam spaces. 

We begin with the definition of amalgam spaces. 
For $p, q \in (0, \infty]$, 
the {\it amalgam space} $(L^p, \ell^q)$ is defined to be 
the set of 
all measurable functions $f$ on $\R^n$ 
such that 
\[
\|f \|_{(L^p, \ell^q)} 
=
\Big\| \big\|
\ichi_{Q} (x-k) f(x)
\big\|_{L^{p}_{x}(\R^n)} 
\Big\|_{\ell^q_{k}(\Z^n)}
=
\bigg\{ \sum_{k \in \Z^n}
\bigg( 
\int_{k+ Q} 
|f(x)|^{p}\, dx 
\bigg)^{q/p}
\bigg\}^{1/q}
< \infty,  
\]
where the representations of $\|\cdot \|_{L^p}$ and 
$\|\cdot \|_{\ell^q}$ need the usual modifications 
if $p=\infty$ or $q=\infty$.  
For properties of amalgam spaces, 
see  
Holland \cite{holland 1975} or 
Fournier--Stewart \cite{fournier stewart 1985}.

For a complex valued $L^1$-function $F$ on the torus 
$\T^n = \R^n/\Z^n$, 
we define its Fourier coefficient by 
\[
\widehat{F}(\mu) = \int_{\T^n} F(x) e^{- 2\pi i \mu \cdot x}\, dx, 
\quad 
\mu \in \Z^n. 
\]
(Although we use the same notation \;$\widehat{}$\;  
to denote both the Fourier coefficient and the Fourier transform, 
we shall use capital letters to denote functions 
on $\T^n$, which will help the reader to distinguish 
the Fourier coefficient from the Fourier transform.) 
For $a\in \ell^{\infty}(\Z^n \times \Z^n)$, 
we define the operator 
$T^{\period }_{a}$ by 
\begin{align*}
&T^{\period }_{a}(F_1,F_2) (x) 
=
\sum_{\mu_1, \mu_2 \in \Z^n} 
 e^{ 2\pi i x\cdot (\mu_1 + \mu_2)}
a(\mu_1, \mu_2) \widehat{F_1}(\mu_1) 
\widehat{F_2}(\mu_2), 
\\
& 
x\in \T^n, 
\quad 
F_1,F_2 \in C^{\infty}(\T^n). 
\end{align*}
For any 
$a\in \ell^{\infty}(\Z^n \times \Z^n)$, 
the operator $T^{\period }_{a}$ is a bilinear mapping 
from $C^{\infty}(\T^n) \times C^{\infty}(\T^n)$ 
to $C^{\infty}(\T^n)$. 
For 
$p_1, p_2, p\in (0, \infty]$, 
we define 
\[
\|T^{\period}_{a}
\|_{L^{p_1}\times L^{p_2}\to L^p}
=
\sup \left\{ 
\frac{
\|T^{\period}_{a} (F_1, F_2)\|_{L^p(\T^n)}}
{\|F_1\|_{L^{p_1}(\T^n)}
\|F_2\|_{L^{p_2}(\T^n)}}
\; \bigg| \; 
F_1, F_2 \in C^{\infty}(\T^n) \setminus \{0\}
\right\}. 
\]

Finally, to give our theorems, 
we need some condition that assures the 
map $a\mapsto \sigma_{a, \Phi}$ to be injective. 
For this we introduce the following: 
we say that a function $\Phi \in C_{0}^{\infty}(\R^d)$ 
satisfies the {\it condition (B)\/}
if there exists a point $\boldsymbol{\xi}^0 \in \R^d$ such that 
\begin{equation*}
\boldsymbol{\xi}^{0} 
\not\in 
\bigcup_{
\boldsymbol{\mu} \in \Z^d
\setminus \{0\}
}
\supp \Phi (\cdot - \boldsymbol{\mu} )
\quad \text{and} \quad 
\Phi (\boldsymbol{\xi}^{0})\neq 0. 
\end{equation*}

Now the following is the first main theorem of this article.

\begin{thmx}\label{th-amalgam-main}
Let $\Phi \in C_{0}^{\infty}(\R^n \times \R^n)$ satisfy 
the condition (B) and 
let 
$p_1, p_2, p$, $q_1, q_2, q \in (0, \infty]$ 
satisfy 
$1/q_1+1/q_2 \ge 1/q$.  
Then there exists a constant $c\in (0, \infty)$ 
depending only on 
$n, p_1,p_2,p, q_1, q_2, q$, and $\Phi$, 
such that 
\[ 
c^{-1} \|T^{\period}_{a}\|_{\pqr} 
\le  \|T_{a,\Phi}\|_{\pqramalgam}
\le c  \|T^{\period }_{a}\|_{\pqr} 
\]
for all 
$a \in \ell^{\infty}(\Z^n\times \Z^n)$. 
\end{thmx}

Before we give the proof of this theorem, 
we give some remarks.

\begin{remx}\label{remark-amalgam} 
(1) The amalgam space $(L^p, \ell^q)$ coincides with 
the Lebesgue space $L^p$ if $p=q$. 
Hence 
the following assertion 
is a special case of Theorem \ref{th-amalgam-main}: 
If $\Phi \in C_{0}^{\infty}(\R^n \times \R^n)$ satisfy 
the condition (B) and if 
$p_1, p_2, p\in (0, \infty]$ 
satisfy 
$1/p_1+1/p_2 \ge 1/p$, 
then there exists a constant $c\in (0, \infty)$ 
depending only on 
$n, p_1,p_2,p$, and $\Phi$, 
such that 
\begin{equation*} 
c^{-1} \|T^{\period}_{a}\|_{\pqr} 
\le  \|T_{a,\Phi}\|_{\pqr}
\le c  \|T^{\period }_{a}\|_{\pqr} 
\end{equation*}
for all 
$a \in \ell^{\infty}(\Z^n\times \Z^n)$, 
where the spaces 
$L^{p_1}, L^{p_2}, L^p$ in the 
quasi-norms of 
$T^{\period}_{a}$ and 
$T_{a,\Phi}$ are the spaces on $\T^n$ and 
on $\R^n$, respectively. 

(2) The latter inequality 
\[ 
\|T_{a,\Phi}\|_{\pqramalgam} 
\le c  \|T^{\period }_{a}\|_{\pqr} 
\] 
in the conclusion of Theorem \ref{th-amalgam-main}  
holds for all $\Phi \in C_{0}^{\infty}(\R^n)$, 
without the condition (B). 
This will be seen from the proof to be given below. 

(3) 
The assumption 
$1/q_1 + 1/q_2 \ge 1/q$ in Theorem \ref{th-amalgam-main} 
gives no essential restriction. 
In fact, 
$T_{\sigma}$ with a nontrivial 
$\sigma \in L^{\infty}(\R^n \times \R^n)$ 
has a finite operator norm in 
$\pqramalgam$ only if $1/q_1 + 1/q_2 \ge 1/q$. 
For a proof of this fact, see 
Lemma \ref{lem-Holder-A} in Appendix.  
\end{remx}

Now we shall proceed to the proof of 
Theorem \ref{th-amalgam-main}. 
The proof 
is a modification of the arguments  
given in \cite{KMT-arxiv, KMT-arxiv-2}. 
We shall divide the proof into 
two parts, 
proof of the latter inequality and 
proof of the former inequality. 
In the proofs, 
$a(\mu_1, \mu_2)$ denotes an arbitrary sequence in 
$\ell^{\infty}(\Z^n \times \Z^n)$. 
We use the letter 
$c$ to denote 
positive constants with 
the same properties as $c$ of the theorem. 
Notice that $c$ in different places may not be 
the same constant.

\begin{proof}[{Proof of the latter inequality 
of Theorem \ref{th-amalgam-main}}] 
Here we shall prove the inequality 
\begin{equation}\label{eq-amalgam-latter} 
\|T_{a,\Phi}\|_{\pqramalgam} 
\le c  \|T^{\period }_{a}\|_{\pqr}.  
\end{equation}
Here we don't need the condition (B).

%

First we follow the methods 
of Coifman--Meyer \cite{CM-AIF, CM-Ast}  
to write $T_{a,\Phi}$ as a superposition of 
simple operators of product forms.

Take a number 
$K\in (0, \infty)$ that satisfies 
$\supp \Phi \subset 2^{-1}KQ \times 2^{-1}KQ$ 
and take a function $\phi$ such that 
\[
\phi \in C_{0}^{\infty}(\R^n ), 
\quad 
\phi (\xi)=1
\;\;\text{on}\;\; 2^{-1}K Q, 
\quad 
\supp \phi \subset 
KQ. 
\]
Since 
$\supp \Phi \subset 2^{-1}KQ \times 2^{-1}KQ$ 
we use the Fourier series expansion 
on 
$KQ \times KQ$ to write $\Phi$ as 
\[
\Phi (\xi_1, \xi_2)
=
\sum_{k_1, k_2  \in \Z^n} 
b(k_1, k_2 ) 
e^{2\pi i K^{-1}(\xi_1 \cdot k_1 + \xi_2 \cdot k_2 )}, 
\quad 
(\xi_1, \xi_2) \in KQ \times KQ, 
\]
where $\{b(k_1, k_2 )\}$ is a rapidly decreasing sequence. 
Multiplying this by 
$\phi (\xi_1)\phi (\xi_2 )$, 
we have 
\begin{align}
\Phi (\xi_1, \xi_2)
&
=
\sum_{k_1, k_2  \in \Z^n} 
b(k_1, k_2 ) 
e^{2\pi i K^{-1}(\xi_1 \cdot k_1 + \xi_2 \cdot k_2 )}
\phi (\xi_1)
\phi (\xi_2 )
\nonumber 
\\
&
=
\sum_{k_1, k_2  \in \Z^n} 
b(k_1, k_2 ) 
( \phi_{k_1}\otimes 
\phi_{k_2}) 
(\xi_1, \xi_2), 
\label{eq-Phi-otimes}
\end{align}
where 
\begin{align*}
&
( \phi_{k_1}\otimes 
\phi_{k_2}) 
(\xi_1, \xi_2)
=
\phi_{k_1}(\xi_1) \phi_{k_2}(\xi_2), 
\\
&
\phi_{k_j}(\xi_j)
= 
e^{2\pi i K^{-1}\xi_j \cdot k_j }
\phi (\xi_j),  
\quad 
j=1,2. 
\end{align*}
Thus 
\begin{align}
\sigma_{a, \Phi} (\xi_1, \xi_2)
&=
\sum_{k_1, k_2}\, 
\sum_{\mu_1, \mu_2}\, 
a(\mu_1, \mu_2)
b(k_1, k_2)  
( \phi_{k_1}\otimes 
\phi_{k_2}) 
(\xi_1 - \mu_1, \xi_2-\mu_2)
\nonumber
\\
&=
\sum_{k_1, k_2}\, 
b(k_1, k_2) \sigma_{ a, 
\phi_{k_1}\otimes 
\phi_{k_2}} 
(\xi_1, \xi_2).  
\label{eq-TaPhi-Taotimes}
\end{align}
Since the sequence $\{ b(k_1, k_2)\} $ is rapidly decreasing, 
in order to prove \eqref{eq-amalgam-latter} 
it is sufficient to prove the 
estimate 
\begin{equation}\label{TaPhik1k2lePa}
\|T_{a, \phi_{k_1}\otimes \phi_{k_2}}\|_{\pqramalgam} 
\le c\,   
\|T^{\period }_{a}\|_{\pqr};  
\end{equation}
recall that $c$ 
should not depend 
on $k_1, k_2$.

Now let $f_1, f_2 \in \calS (\R^n)$. 
To calculate 
the $(L^p, \ell^q)$-quasi-norm of a function, 
it is convenient to write the variables of $\R^n$ 
as $x+\rho$ with $x\in Q$ and $\rho \in \Z^n$. 
Thus let 
$x\in Q$ and $\rho\in \Z^n$.
We have 
\begin{align*}
&
T_{  a, \phi_{k_1} \otimes \phi_{k_2} }
(f_1,f_2)(x+\rho)
\\
&=
\sum_{\mu_1, \mu_2\in \Z^n}\,  
\iint_{\xi_1, \xi_2 \in \R^n} 
a(\mu_1, \mu_2) 
e^{2\pi i (x+\rho)\cdot (\xi_1 + \xi_2)} 
\\
&
\quad \quad \quad 
\times 
\phi_{k_1}(\xi_1 - \mu_1) 
\phi_{k_2}(\xi_2 - \mu_2) 
\widehat{f_1}(\xi_1)
\widehat{f_2}(\xi_2)\, 
d\xi_1 d\xi_2
\\
&
=
\sum_{\mu_1, \mu_2\in \Z^n}\, 
\iint_{\xi_1, \xi_2 \in \R^n} 
a(\mu_1, \mu_2) 
e^{2\pi i (x+\rho)\cdot (\xi_1 +\mu_1 + \xi_2+\mu_2)} 
\\
&
\quad \quad \quad 
\times 
\phi_{k_1}(\xi_1) 
\phi_{k_2}(\xi_2) 
\widehat{f_1}(\xi_1 + \mu_1)
\widehat{f_2}(\xi_2 + \mu_2)\, 
d\xi_1 d\xi_2
\\
&
=(\ast). 
\end{align*}
Here notice that 
$e^{2\pi i \rho \cdot (\mu_1 + \mu_2)}=1$ 
since 
$\rho \cdot (\mu_1 + \mu_2)$ are integers. 
We write 
\begin{align*}
&
e^{2\pi i (x+\rho) \cdot (\xi_1 + \mu_1 + \xi_2 + \mu_2)} 
=
e^{2\pi i x \cdot (\xi_1 + \xi_2)}
e^{2\pi i x \cdot (\mu_1 + \mu_2)}
e^{2\pi i \rho \cdot (\xi_1+\xi_2)}  
\\
&=
e^{2\pi i x \cdot (\mu_1 + \mu_2)}
e^{2\pi i \rho \cdot \xi_1 }
e^{2\pi i \rho \cdot \xi_2}
\sum_{\alpha}
\frac{1}{\alpha !} (2\pi i)^{|\alpha|} 
x^{\alpha}\xi_1^{\alpha}
\sum_{\beta}
\frac{1}{\beta !} (2\pi i)^{|\beta|} 
x^{\beta}\xi_2^{\beta}, 
\end{align*}
where the sums are taken 
over all multi-indices $\alpha$ and $\beta$.  
Thus 
\begin{align*}
(\ast)
&=
\sum_{\mu_1, \mu_2} 
\sum_{\alpha, \beta} 
a(\mu_1, \mu_2)
\frac{(2\pi i)^{|\alpha|} }{\alpha !} \frac{(2\pi i)^{|\beta|} } {\beta !} 
x^{\alpha+\beta}
e^{2\pi i x \cdot (\mu_1 + \mu_2)}
\\
&\quad 
\times 
\bigg( \int_{\R^n} 
e^{2\pi i \rho \cdot \xi_1} 
\phi_{k_1} (\xi_1)\, \xi_1^{\alpha}\, 
\widehat{f_1}(\xi_1+ \mu_1)\, d\xi_1\bigg)
\bigg( \int_{\R^n} 
e^{2\pi i \rho \cdot \xi_2} 
\phi_{k_2} (\xi_2)\, \xi_2^{\beta}\,  
\widehat{f_2}(\xi_2+\mu_2)\, d\xi_2\bigg)
\\
&
=(\ast\ast). 
\end{align*}
We define 
$F_{k_1, \rho, \alpha}, G_{k_2, \rho, \beta} \in C^{\infty}(\T^n)$ 
so that their Fourier coefficients are given by  
\begin{align*}
&
(F_{k_1, \rho, \alpha})^{\wedge}(\mu)
=\int_{\R^n} 
e^{2\pi i \rho \cdot \xi_1} 
\phi_{k_1} (\xi_1)\, \xi_1^{\alpha}\,  
\widehat{f_1}(\xi_1+ \mu)\, d\xi_1, 
\quad \mu \in \Z^n, 
\\
&
(G_{k_2, \rho, \beta})^{\wedge}(\mu)
=\int_{\R^n} 
e^{2\pi i \rho \cdot \xi_2 } 
\phi_{k_2} (\xi_2)\,  \xi_2^{\beta}\,  
\widehat{f_2}(\xi_2 + \mu)\, d\xi_2, 
\quad \mu \in \Z^n. 
\end{align*}
Then 
\begin{align*}
(\ast\ast)
&=
\sum_{\mu_1, \mu_2} 
\sum_{\alpha, \beta} 
a(\mu_1, \mu_2)
\frac{(2\pi i)^{|\alpha|} }{\alpha !} \frac{(2\pi i)^{|\beta|} } {\beta !} 
x^{\alpha+\beta}
e^{2\pi i x \cdot (\mu_1 + \mu_2)}
(F_{k_1, \rho, \alpha})^{\wedge}(\mu_1)
(G_{k_2, \rho, \alpha})^{\wedge}(\mu_2)
\\
&
=
\sum_{\alpha, \beta} 
\frac{(2\pi i)^{|\alpha|} }{\alpha !} \frac{(2\pi i)^{|\beta|} } {\beta !} 
x^{\alpha+\beta}\, 
T^{\period }_{a}(F_{k_1, \rho, \alpha}, G_{k_2, \rho, \alpha})(x). 
\end{align*}
Thus we obtain 
\begin{align*}
&T_{a, \phi_{k_1} \otimes \phi_{k_2} }
(f_1,f_2) (x+ \rho ) 
\\
&=
\sum_{\alpha, \beta} 
\frac{(2\pi i)^{|\alpha|} }{\alpha !} \frac{(2\pi i)^{|\beta|} } {\beta !} 
x^{\alpha+\beta}\, 
T^{\period }_{a}(F_{k_1, \rho, \alpha}, G_{k_2, \rho, \alpha})(x), 
\quad 
x \in Q, \quad 
\rho \in \Z^n. 
\end{align*}

From the last formula, we have 
\begin{align*}
&
\| 
T_{a, \phi_{k_1} \otimes \phi_{k_2} }
(f_1,f_2)
\|_{(L^{p}, \ell^{q})}
=
\Big\| \big\| 
T_{a, \phi_{k_1} \otimes \phi_{k_2} } (f_1,f_2)(x+\rho)
\big\|_{L^{p}_{x}(Q)}
\Big\|_{
\ell^{q}_{\rho}(\Z^n)}
\\
&=
\Bigg\| 
\bigg\| 
\sum_{\alpha, \beta} 
\frac{(2\pi i)^{|\alpha|} }{\alpha !} \frac{(2\pi i)^{|\beta|} } {\beta !} 
x^{\alpha+\beta}
\, 
T^{\period }_{a}(F_{k_1, \rho, \alpha}, G_{k_2, \rho, \alpha})(x)
\bigg\|_{L^{p}_{x}(Q)} 
\Bigg\|_{\ell^{q}_{\rho}(\Z^n)}
\\
&
\le 
\bigg\{ 
\sum_{\alpha, \beta} 
\bigg(
\frac{(2\pi)^{|\alpha|} }{\alpha !} \frac{(2\pi)^{|\beta|} } {\beta !} 
\Big\|
\big\| T^{\period }_{a}
(F_{k_1, \rho, \alpha}, G_{k_2, \rho, \alpha})(x)
\big\|_{L^{p}_{x}(Q)} 
\Big\|_{\ell^{p}_{\rho}(\Z^n)}
\bigg)^{\epsilon}
\bigg\} ^{1/\epsilon}
\\
&=
(\ast\ast\ast)
\end{align*}
with $\epsilon= \min \{p, q, 1\}$. 
We set 
$1/q_1+1/q_2=1/s$. 
Our assumption implies 
$1/s\ge 1/q$ and hence  
the embedding 
$\ell^{s}\hookrightarrow \ell^{q}$ holds. 
Thus, 
the definition of 
$\|T^{\period }_{a}\|_{\pqr}$, 
the embedding 
$\ell^{s}\hookrightarrow \ell^{q}$, 
and H\"older's inequality with exponents $1/q_1+1/q_2=1/s$ 
yield   
\begin{align*}
&
\Big\| 
\big\| T^{\period }_{a}
(F_{k_1, \rho, \alpha}, G_{k_2, \rho, \alpha})(x)
\big\|_{L^{p}_{x}(Q)} 
\Big\|_{\ell^{q}_{\rho}(\Z^n)}
\\
&
\le 
\Big\| 
\big\|T^{\period }_{a}\big\|_{\pqr}  
\big\| F_{k_1, \rho, \alpha} \big\|_{L^{p_1} (Q)}
\big\|G_{k_2, \rho, \beta }\big\|_{L^{p_2} (Q)}
\Big\|_{\ell^{q}_{\rho}(\Z^n)}
\\
&
\le 
\Big\| 
\big\|T^{\period }_{a}\big\|_{\pqr}  
\big\| F_{k_1, \rho, \alpha} \big\|_{L^{p_1} (Q)}
\big\|G_{k_2, \rho, \beta }\big\|_{L^{p_2} (Q)}
\Big\|_{\ell^{s}_{\rho}(\Z^n)}
\\
&
\le 
\big\|T^{\period }_{a}\big\|_{\pqr} 
\Big\| 
\big\| F_{k_1, \rho, \alpha} \big\|_{L^{p_1} (Q)} 
\Big\|_{\ell^{q_1}_{\rho}(\Z^n)}
\Big\| 
\big\|G_{k_2, \rho, \beta}\big\|_{L^{p_2} (Q)} 
\Big\|_{\ell^{q_2}_{\rho}(\Z^n)}. 
\end{align*}
Hence 
\begin{align*}
(\ast\ast\ast)
&
\le 
\bigg\{ 
\sum_{\alpha, \beta} 
\bigg(
\frac{(2\pi)^{|\alpha|} }{\alpha !} \frac{(2\pi)^{|\beta|} } {\beta !}\, 
\big\|T^{\period }_{a}\big\|_{\pqr}  
\\
&
\quad \times 
\Big\| 
\big\| F_{k_1, \rho, \alpha} \big\|_{L^{p_1} (Q)} 
\Big\|_{\ell^{q_1}_{\rho}(\Z^n)}
\Big\| 
\big\|G_{k_2, \rho, \alpha}\big\|_{L^{p_2} (Q)} 
\Big\|_{\ell^{q_2}_{\rho}(\Z^n)}
\bigg)^{\epsilon}
\bigg\} ^{1/\epsilon}. 
\end{align*}
Thus, if we 
prove the estimates  
\begin{align}
&
\Big\| 
\big\| 
F_{k_1, \rho, \alpha} \big\|_{L^{p_1} (Q)} 
\Big\|_{\ell^{q_1}_{\rho}(\Z^n)}
\le c\,  
\langle \alpha \rangle^N  (1+K)^{|\alpha|}
\|f_1\|_{(L^{p_1}, \ell^{q_1}) }, 
\label{Ff}
\\
&
\Big\| 
\big\|
G_{k_2, \rho, \beta } \big\|_{L^{p_2} (Q)} 
\Big\|_{\ell^{q_2}_{\rho}(\Z^n)}
\le c\,  
\langle \beta \rangle^N  (1+K)^{|\beta|}
\|f_2\|_{(L^{p_2}, \ell^{q_2}) }  
\label{Gg}
\end{align}
with $N$ depending only on 
$n, p_1, p_2, q_1, q_2$,  
then we obtain 
\begin{align*}
&
\| T_{a, \phi_{k_1} \otimes \phi_{k_2}}
(f_1,f_2)
\|_{(L^{p}, \ell^{q})}
\\
&\le 
c\, 
\|T^{\period }_{a}\|_{\pqr} 
 \|f_1\|_{(L^{p_1}, \ell^{q_1})}
\|f_2\|_{(L^{p_2}, \ell^{q_2})}
\\
&\quad 
\times 
\bigg\{ 
\sum_{\alpha, \beta} 
\bigg(
\frac{(2\pi)^{|\alpha|} }{\alpha !} \frac{(2\pi)^{|\beta|} } {\beta !} 
\langle \alpha \rangle^{N}  (1+K)^{|\alpha|}
\langle \beta \rangle^{N}  (1+K)^{|\beta|}
\bigg)^{\epsilon}
\bigg\} ^{1/\epsilon}
\\
&
= 
c\, 
\|T^{\period }_{a}\|_{\pqr} 
 \|f_1\|_{(L^{p_1}, \ell^{q_1})}
\|f_2\|_{(L^{p_2}, \ell^{q_2})}, 
\end{align*}
which is the desired estimate \eqref{TaPhik1k2lePa}.

Thus our task is to prove \eqref{Ff} and \eqref{Gg}. 
By symmetry, it is sufficient to prove one of them. 
We shall prove \eqref{Ff}. 
Here 
we use the Poisson summation formula 
\[
\sum_{\mu\in \Z^n} 
e^{2\pi i \mu \cdot x}
\widehat{f_1}(\xi_1+ \mu)
=
\sum_{\nu \in \Z^n} 
e^{-2\pi i \xi_1 \cdot ( x+ \nu ) }
f_1 (x+ \nu )  
\]
(for this formula, 
see for example 
\cite[Chapter VII, Section 2]{stein-weiss} 
or 
\cite[Section 3.2.3]{grafakos 2014c}). 
Using this formula, 
we can write $F_{k_1, \rho, \alpha}(x) $ 
as 
\begin{align*}
F_{k_1, \rho, \alpha}(x)
&=
\sum_{\mu\in \Z^n} 
(F_{k_1, \rho, \alpha})^{\wedge} (\mu) 
e^{2\pi i \mu \cdot x}
\\
&
=
\sum_{\mu\in \Z^n} 
e^{2\pi i \mu \cdot x}
\int_{\R^n} 
e^{2\pi i \rho \cdot \xi_1} 
\phi_{k_1} (\xi_1)\, \xi_1^{\alpha}\,  
\widehat{f_1}(\xi_1+ \mu)\, d\xi_1
\\
&
=
\sum_{\nu \in \Z^n} 
\int_{\R^n} 
e^{2\pi i \rho \cdot \xi_1} 
\phi_{k_1} (\xi_1)
\, \xi_1^{\alpha}\,  
e^{-2\pi i \xi_1 \cdot (x+ \nu ) }
f_1 (x+ \nu )
\, d \xi_1
\\
&
=
\sum_{\nu \in \Z^n} 
\calF^{-1}\big(
\phi_{k_1} (\xi_1) \, \xi_1^{\alpha} 
\big) (\rho - x- \nu )
f_1 (x+ \nu ). 
\end{align*}
For $x \in Q$ and 
for each $N\in \N$, 
we have 
\begin{align*}
&
\big| \calF^{-1}\big(
\phi_{k_1} (\xi_1) \xi_1^{\alpha} 
\big) (\rho - x- \nu )
\big|
=
\big| \calF^{-1}\big(
e^{2\pi i K^{-1}k_1 \cdot \xi_1 }
\phi (\xi_1) \xi_1^{\alpha} 
\big) (\rho - x- \nu )
\big|
\\
&=
\big| 
\calF^{-1}\big(
\phi (\xi_1) \xi_1^{\alpha} 
\big) (K^{-1}k_1+\rho - x- \nu )
\big|
\\
&
\le c_{n, N}\, 
\sup_{|\gamma|\le N} 
\big\|\partial_{\xi_1}^{\gamma} 
\big(\phi (\xi_1) \xi_1^{\alpha} \big)
\big\|_{L^1_{\xi_1}}
\langle K^{-1}k_1 +\rho - x-\nu \rangle^{-N}
\\
&
\le c_{n, N, \phi}\,  
\langle \alpha \rangle^{N} (1+K)^{|\alpha|}
\langle K^{-1}k_1 + \rho - \nu \rangle^{-N}
\end{align*}
and hence 
\begin{align*}
|F_{k_1, \rho, \alpha}(x)|
&\le 
c_{n, N, \phi}\,  
\langle \alpha \rangle^{N} (1+K)^{|\alpha|} 
\sum_{\nu\in \Z^n } 
\langle  K^{-1}k_1+ \rho- \nu \rangle^{-N}
| f_1 (x+ \nu ) | 
\\
&
=
c_{n, N, \phi}\,  
\langle \alpha \rangle^{N} (1+K)^{|\alpha|} 
\sum_{\sigma \in \Z^n} 
\langle  K^{-1}k_1+\sigma \rangle^{-N}
| f_1 (x+ \rho - \sigma )|. 
\end{align*}
Set 
$\epsilon_1=\min \{1, p_1, q_1\}$ 
and choose $N\in \N$ so that 
$\epsilon_1 N > n$. 
Then 
\begin{align*}
&
\Big\|
\big\|
F_{k_1, \rho, \alpha}(x)\big\|_{L^{p_1}_x (Q)} 
\Big\|_{\ell^{q_1}_{\rho}(\Z^n)}
\\
&
\le 
c\,  
\langle \alpha \rangle^{N} (1+K)^{|\alpha|} \Bigg\| 
\bigg\| 
\sum_{\sigma\in \Z^n} 
\langle K^{-1}k_1 +\sigma \rangle^{-N}
| f_1 (x +  \rho - \sigma) | 
\bigg\|_{L^{p_1}_x (Q)} 
\Bigg\|_{\ell^{q_1}_{\rho}(\Z^n)}
\\
&
\le 
c\,  
\langle \alpha \rangle^{N} (1+K)^{|\alpha|} 
\\
&\quad 
\times 
\bigg( 
\sum_{\sigma\in \Z^n} 
\langle K^{-1}k_1 +\sigma \rangle^{- \epsilon_1 N}
\Big\| 
\big\|
f_1 (x + \rho - \sigma)
\big\|_{L^{p_1}_x (Q)} 
\Big\|_{\ell^{q_1}_{\rho}(\Z^n)} ^{\epsilon_1}
\bigg)^{1/\epsilon_1}
\\
&
= 
c\,  
\langle \alpha \rangle^{N} (1+K)^{|\alpha|} 
\bigg( \sum_{\sigma\in \Z^n} 
\langle K^{-1}k_1 + \sigma \rangle^{-\epsilon_1 N}
\bigg)^{1/\epsilon_1 }
\|f_1\|_{(L^{p_1}, \ell^{q_1}) }
\\
&
\le   
c\,  
\langle \alpha \rangle^{N} (1+K)^{|\alpha|} 
\|f_1\|_{(L^{p_1}, \ell^{q_1}) (\R^n)},  
\end{align*}
which implies \eqref{Ff}. 
Now 
the latter inequality of Theorem 
\ref{th-amalgam-main} is proved. 
\end{proof}

Next, we prove 
the former inequality of Theorem \ref{th-amalgam-main}.

\begin{proof}[{Proof of the former inequality of 
Theorem \ref{th-amalgam-main}}] 
Here we shall prove the inequality 
\begin{equation}\label{eq-amalgam-former} 
\|T^{\period }_{a}\|_{\pqr} 
\le  c \|T_{a,\Phi}\|_{\pqramalgam}. 
\end{equation}

From the assumption that 
$\Phi $ satisfies the  condition (B), 
there exist a point 
$\boldsymbol{\xi}^{0}
\in \R^{2n}$ 
and a sufficiently small $\epsilon>0$ such that 
\begin{align}
&
\Phi (\boldsymbol{\xi}^{0}) \neq 0, 
\label{Phineq0}
\\
&
|\boldsymbol{\xi}- \boldsymbol{\xi}^{0}|< 2 \epsilon, 
\; 
\boldsymbol{\mu}  \in \Z^{2n}, 
\; 
\boldsymbol{\mu}  \neq 0   
\; 
\Rightarrow 
\; 
\Phi (\boldsymbol{\xi} - \boldsymbol{\mu}) = 0. 
\label{Phisupp}
\end{align}
We write $\boldsymbol{\xi}^{0}=(\xi^0_1, \xi^0_2)$. 
We then take 
functions 
$\theta_1, \theta_2\in C_{0}^{\infty}(\R^n)$ such that 
\begin{align}
&
\supp \theta_j \subset 
\{  \xi \in \R^n 
\mid 
|\xi - \xi^{0}_{j}| < \epsilon \}, 
\label{thetajsupp}
\\
&
\bigg|
\iint_{\R^n \times \R^n} 
e^{2\pi i x\cdot (\xi_1+\xi_2)}
\Phi (\xi_1, \xi_2) 
\theta_1 (\xi_1)
\theta_2 (\xi_2)
\, d\xi_1 d\xi_2
\bigg|
\ge 1
\;\;\text{for all}\;\; x\in Q.  
\label{thetaPhi}
\end{align}
Hereafter we write 
\begin{equation}\label{eq-def-g}
g(x)=\iint_{\R^n \times \R^n} 
e^{2\pi i x\cdot (\xi_1+\xi_2)}
\Phi (\xi_1, \xi_2) 
\theta_1 (\xi_1)
\theta_2 (\xi_2)
\, d\xi_1 d\xi_2. 
\end{equation}

Take arbitrary $F_1, F_2 \in C^{\infty}(\T^n)$. 
We define 
$f_1, f_2  \in \calS (\R^n)$ so that 
their Fourier transforms are given by  
\begin{equation*}
\widehat{f_j}(\xi) = 
\sum_{\nu \in \Z^n} 
\widehat{F_j}(\nu ) 
\theta_j (\xi - \nu ),
\quad \xi \in \R^n, 
\quad 
j=1,2, 
\end{equation*}
or equivalently by 
\begin{equation*}
f_j (x) =  \sum_{\nu \in \Z^n} 
\widehat{F_j}(\nu ) 
e^{2\pi i \nu \cdot x}
(\calF^{-1}\theta_j) (x)
=
F_j (x) (\calF^{-1}\theta_j) (x), 
\quad j=1,2  
\end{equation*}
(recall that $\widehat{F_j}$ denotes the Fourier coefficient 
of $F_j$).  
Then, since $F_j$ is a periodic function 
and $\calF^{-1}\theta_j$ is a function in $\calS$, 
we have 
\begin{equation}\label{fF}
\|f_j\|_{(L^{p_j}, \ell^{q_j}) (\R^n)} \le c 
\|F_j\|_{L^{p_j}(\T^n)}, 
\quad 
j=1,2. 
\end{equation}

On the other hand, 
from 
\eqref{Phisupp} and \eqref{thetajsupp}, 
we have 
\begin{align*}
&
\sigma_{a, \Phi} (\xi_1, \xi_2) \widehat{f_1}(\xi_1) 
\widehat{f_2}(\xi_2) 
\\
&
=
\bigg(\sum_{\mu_1, \mu_2 \in \Z^n} 
a(\mu_1, \mu_2) \Phi (\xi_1 - \mu_1, \xi_2 - \mu_2)
\bigg)
\\
&
\quad \times 
\bigg( 
\sum_{\nu_1 \in \Z^n} 
\widehat{F_1}(\nu_1) 
\theta_1 (\xi_1 - \nu_1)
\bigg)
\bigg( 
\sum_{\nu_2 \in \Z^n} 
\widehat{F_2}(\nu_2) 
\theta_2 (\xi_2 - \nu_2) 
\bigg)
\\
&
=
\sum_{\mu_1, \mu_2 \in \Z^n} 
a(\mu_1, \mu_2) 
\widehat{F_1}(\mu_1) 
\widehat{F_2}(\mu_2) 
\Phi  (\xi_1 - \mu_1, \xi_2 - \mu_2 ) 
\theta_1 (\xi_1 - \mu_1 ) \theta_2 (\xi_2 - \mu_2)
\end{align*}
and thus 
\begin{align}
T_{a,\Phi} (f_1,f_2)(x)
&=
\sum_{\mu_1, \mu_2 \in \Z^n} 
a(\mu_1, \mu_2) 
\widehat{F_1}(\mu_1) 
\widehat{F_2}(\mu_2) 
\iint_{\R^n \times \R^n} 
e^{2\pi i x \cdot (\xi_1 + \xi_2)}
\nonumber 
\\
&
\quad \times 
\Phi  (\xi_1 - \mu_1, \xi_2 - \mu_2 ) 
\theta_1 (\xi_1 - \mu_1) \theta_2 (\xi_2 - \mu_2) \, d\xi_1 d\xi_2 
\nonumber 
\\
&
=
\sum_{\mu_1, \mu_2 \in \Z^n} 
a(\mu_1, \mu_2) 
\widehat{F_1}(\mu_1) 
\widehat{F_2}(\mu_2) 
e^{2\pi i x \cdot (\mu_1 + \mu_2)}
g(x) 
\label{eq-iii}
\\
&
=
T^{\period }_{a}(F_1,F_2)(x) 
g(x). 
\nonumber 
\end{align}
From this and 
\eqref{thetaPhi}-\eqref{eq-def-g}, we have 
\begin{equation}\label{TgeTperiod}
|T_{a,\Phi} (f_1,f_2)(x)| 
\ge  
|T^{\period }_{a}(F_1,F_2)(x) |, 
\quad 
x \in Q.   
\end{equation}

Now from 
\eqref{TgeTperiod} and \eqref{fF}, 
we obtain 
\begin{align*}
&
\|T^{\period }_{a}(F_1,F_2) \|_{L^{p} (Q)}
\le  
\|T_{a,\Phi}(f_1,f_2)\|_{L^{p} (Q)}
\le 
\|T_{a,\Phi}(f_1,f_2)\|_{(L^{p}, \ell^{q}) (\R^n)}
\\
&
\le 
\|T_{a,\Phi}\|_{\pqramalgam} 
\|f_1\|_{(L^{p_1}, \ell^{q_1})(\R^n)}
\|f_2\|_{(L^{p_2}, \ell^{q_2})(\R^n)}
\\
&
\le c   
\|T_{a,\Phi}\|_{\pqramalgam} 
\|F_1\|_{L^{p_1}(\T^n)}
\|F_2\|_{L^{p_2}(\T^n)},  
\end{align*}
which implies 
\eqref{eq-amalgam-former}. 
Now the former inequality of 
Theorem \ref{th-amalgam-main} is proved 
and proof of Theorem \ref{th-amalgam-main} is complete. 
\end{proof}


\section{The case of Wiener amalgam spaces}
\label{Wiener-amalgam}

In this section, we shall give our 
second main theorem, which concerns 
the operator norm of $T_{a, \Phi}$ in Wiener amalgam spaces.

We begin with the definition of Wiener amalgam spaces. 
Let $\kappa\in C_{0}^{\infty}(\R^n)$ be a function satisfying 
\begin{equation*}
\bigg| \sum _{k\in \mathbb{Z}^{n}} \kappa(\xi-k) \bigg| 
\geq 1
\quad \text{for all} \quad
\xi \in \R^n .
\end{equation*}
Then  for $p,q \in (0, \infty]$,
the {\it Wiener amalgam space} $W^{p,q}=W^{p,q}(\R^n)$
is defined to be the set of
all $f \in \calS'(\R^n)$ such that
\begin{equation*}
\|f\|_{W^{p,q}} = 
\Big\| \big\|
\kappa(D-k)f(x)
\big\|_{\ell^{q}_{k}(\Z^n)} 
\Big\|_{L^p_{x}(\R^n)}
< \infty.
\end{equation*}
It is known that the definition of Wiener amalgam space 
does not depend on the choice of the function $\kappa$ 
up to the equivalence of quasi-norm. 
It is also known that 
the embedding 
$W^{p_1, q_1} \hookrightarrow W^{p_2, q_2} $ 
holds if $0<p_1 \le p_2\le \infty$ 
and $0<q_1 \le q_2\le \infty$. 
For these facts, see Feichtinger \cite{Fei-1981, Fei-1983}, 
and 
Triebel \cite{Tri-1983}.

%
We 
write $X(\Z^n)$ to denote the set of all functions  
$b: \Z^n \to \C$ such that $b(\mu)=0$ except for 
a finite number of $\mu \in \Z^n$. 
For $a\in \ell^{\infty}(\Z^n \times \Z^n)$, 
we define the bilinear 
map $S_a: X(\Z^n)\times X(\Z^n) \to X(\Z^n)$ by 
\begin{align*}
&
S_a(b_1, b_2)(\mu)
=
\sum_{\mu_1+\mu_2=\mu} 
a(\mu_1, \mu_2)
b_1 (\mu_1)
b_2 (\mu_2), 
\\
&
\mu \in \Z^n, 
\quad 
b_1, b_2 \in X(\Z^n). 
\end{align*}
For $q_1, q_2, q\in (0, \infty]$, 
we define 
\[
\|S_a\|_{\ell^{q_1}\times \ell^{q_2}\to \ell^q}
=
\sup \left\{ 
\frac{
\|S_a(b_1, b_2)\|_{\ell^q(\Z^n)}}
{\|b_1\|_{\ell^{q_1}(\Z^n)}
\|b_2\|_{\ell^{q_2}(\Z^n)}}
\; \bigg| \; 
b_1, b_2 \in X(\Z^n)\setminus \{0\}
\right\}. 
\]

The following is the second main theorem of this article.

\begin{thmx}\label{th-Wiener-amalgam-main}
Let $\Phi \in C_{0}^{\infty}(\R^n \times \R^n)$ satisfy 
the condition (B) and let 
$p_1, p_2, p$, $q_1, q_2, q\in (0, \infty]$ 
satisfy  
$1/p_1 + 1/p_2 \ge 1/p$. 
Then 
there exists a constant 
$c\in (0, \infty)$ 
depending only on 
$n, p_1, p_2, p, q_1, q_2, q$, and $\Phi$, 
such that 
\[
c^{-1} 
\|S_a\|_{\ell^{q_1}\times \ell^{q_2}\to \ell^q}
\le  
\left\|
T_{a, \Phi }
\right\|_{W^{p_1, q_1} \times W^{p_2, q_2} \to W^{p,q} }
\le 
c \|S_a\|_{\ell^{q_1}\times \ell^{q_2}\to \ell^q} 
\]
for all $a \in \ell^{\infty}(\Z^n \times \Z^n)$. 
\end{thmx}

Before we give the proof of this theorem, 
we give some remarks.

\begin{remx}\label{remark-Wiener-amalgam}  
(1) The latter inequality 
\[ 
\left\|
T_{a, \Phi }
\right\|_{W^{p_1, q_1} \times W^{p_2, q_2} \to W^{p,q} }
\le 
c \|S_a\|_{\ell^{q_1}\times \ell^{q_2}\to \ell^q} 
\] 
in the conclusion of Theorem \ref{th-Wiener-amalgam-main} 
holds for all $\Phi \in C_{0}^{\infty}(\R^n)$, 
without the condition (B). 
This will be seen from the proof to be given below. 

(2) The assumption 
$1/p_1 + 1/p_2 \ge 1/p$ in Theorem \ref{th-Wiener-amalgam-main} 
gives no essential restriction. 
$T_{\sigma}$ with a nontrivial 
$\sigma \in L^{\infty}(\R^n \times \R^n)$ 
has a finite operator norm in 
$W^{p_1, q_1} \times W^{p_2, q_2} \to W^{p,q} $ 
only if $1/p_1 + 1/p_2 \ge 1/p$. 
For a proof of this fact, see 
Lemma \ref{lem-Holder-WA} in Appendix.  \end{remx}

In the proof of the latter inequality of 
Theorem \ref{th-Wiener-amalgam-main}, 
we use the following lemma.

\begin{lmmx}\label{varphigmu}
Let 
$\varphi \in C_{0}^{\infty}(\R^n)$ and let 
$g_{\mu}\in \calS^{\prime}(\R^n)$, $\mu \in \Z^n$. 
Suppose 
the Fourier transform of each $g_{\mu}$ has 
a compact support 
and suppose 
there exists a number $K\in (0, \infty)$ 
such that 
$\diam (\supp \varphi) \le K$ 
and 
$\diam (\supp \widehat{g_{\mu}}) \le K$ 
for all $\mu \in \Z^n$. 
Then 
for each 
$p, q\in (0, \infty]$ 
there exists a constant 
$c$ depending only on 
$n, p, q, K$, and $\varphi$ 
such that  
\begin{equation*} 
\Big\| 
\big\| 
\varphi (D - \mu) g_{\mu}(x) \big\|_{\ell^q_{\mu}(\Z^n)} 
\Big\|_{L^p_{x}(\R^n)}
\le 
c
\Big\| 
\big\|  g_{\mu}(x) \big\|_{\ell^q_{\mu}(\Z^n)} 
\Big\|_{L^p_x(\R^n)}.  
\end{equation*}
\end{lmmx} 

\begin{proof}
We use the following two well known facts.  
Firstly, 
if the Fourier transform of 
$f\in \calS^{\prime}(\R^n)$ has a compact support 
and if $R$ is a positive real number satisfying 
$\diam (\supp \widehat{f}) \le R$ 
then for each 
$r,s$ satisfying 
$0< r \le s \le \infty$, 
there exists a constant 
$c$ depending only on $r, s$, and $n$ such that 
\begin{equation}\label{bbb}
R^{n/s} \|f\|_{L^s(\R^n)} \le 
c R^{n/r} \|f\|_{L^r(\R^n)}.  
\end{equation}
For a proof of this inequality, see for example 
\cite[Proposition 1.3.2]{Triebel-book 1983}. 
Secondly, 
\begin{equation}\label{eqmixednorm}
\Big\| \| f(x,y) \|_{L^p_x} 
\Big\|_{L^q_y}
\le 
\Big\| \| f(x,y) \|_{L^q_y} 
\Big\|_{L^p_x} 
\quad \text{if} \quad 
0<p\le q \le \infty,  
\end{equation}
which holds for all 
$L^p$ and $L^q$ quasi-norms 
defined on any $\sigma$-finite 
measure spaces. 
The inequality \eqref{eqmixednorm} 
can be easily proved by the use of 
Minkowski's inequality for integrals.

Now let $\varphi$ and $g_{\mu}$ be as in Lemma \ref{varphigmu}.   
We write 
\[
\varphi (D-\mu) g_{\mu} (x) 
=
\int 
e^{2\pi i \mu \cdot y} 
(\calF^{-1}\varphi) (y) 
g_{\mu}(x- y)\, dy. 
\]
From our assumption, 
the Fourier transform 
of the function 
$y \mapsto  
(\calF^{-1}\varphi) (y) 
g_{\mu}(x- y)$ 
has a compact support 
of diameter not exceeding $2K$. 
Thus by \eqref{bbb} 
we have 
\begin{align*}
&
|\varphi (D-\mu) g_{\mu} (x)|
=\bigg| 
\int_{\R^n}
e^{2 \pi i \mu\cdot y} 
(\calF^{-1}\varphi) (y) 
g_{\mu}(x- y)\, dy
\bigg|
\\
&
\le 
\int_{\R^n}
\big|(\calF^{-1}\varphi) (y) 
g_{\mu}(x-y)\big|\, dy
\le 
c_{n, \epsilon, K}   
\big\| 
(\calF^{-1}\varphi) (y) 
g_{\mu}(x-y)
\big\|_{L^{\epsilon}_{y}}
\end{align*}
for any $\epsilon$ satisfying $0<\epsilon \le 1$. 
Taking $\epsilon$ so that 
$\epsilon \le  \min \{1, p,q\}$, 
we use \eqref{eqmixednorm} to obtain 
\begin{align*}
&
\Big\|
\big\| 
\varphi (D-\mu) g_{\mu} (x)\big\|_{\ell^q_{\mu}} 
\Big\|_{L^p_x}
\\
&
\le
c_{n, \epsilon, K}
\Big\|
\big\| 
\| 
(\calF^{-1}\varphi) (y) 
g_{\mu}(x-y)
\|_{L^{\epsilon}_{y}}
\big\|_{\ell^q_{\mu}} 
\Big\|_{L^p_x}
\\
&
\le 
c_{n, \epsilon, K}  
\Big\| 
\big\|
\| 
(\calF^{-1}\varphi) (y) 
g_{\mu}(x-y)
\|_{\ell^q_{\mu}} 
\big\|_{L^p_x}
\Big\|_{L^{\epsilon}_{y}}
=
c_{n, \epsilon, K} 
\big\| \calF^{-1}\varphi \big\|_{L^{\epsilon}}
\Big\|
\big\| 
g_{\mu} (x)\big\|_{\ell^q_{\mu}} 
\Big\|_{L^p_x}. 
\end{align*}
Lemma \ref{varphigmu} is proved. 
\end{proof}

Now we shall prove 
Theorem \ref{th-Wiener-amalgam-main}. 
The proof 
is a modification of the argument 
given in \cite{KMT-arxiv-3}. 
We shall divide the proof  
into two parts,  
proof of the latter inequality 
and proof 
of the former inequality. 
In the proofs, we 
assume 
$a\in \ell^{\infty}(\Z^n \times \Z^n)$. 
For nonnegative quantities $A$ and $B$, 
we write 
$A\lesssim B$ if 
there exists a constant $c$ with the 
same properties as the constant 
$c$ of Theorem \ref{th-Wiener-amalgam-main}. 
Also we write $A \approx B$ to mean that 
$A \lesssim B$ and $B \lesssim A$.

\begin{proof}[{Proof of the latter inequality of 
Theorem \ref{th-Wiener-amalgam-main}}] 
Here we shall prove the inequality 
\begin{equation}\label{eq-Wiener-amalgam-latter}
\left\|
T_{a, \Phi }
\right\|_{W^{p_1, q_1} \times W^{p_2, q_2} \to W^{p,q} }
\le 
c \|S_a\|_{\ell^{q_1}\times \ell^{q_2}\to \ell^q}. 
\end{equation}
Here we don't need the condition (B). 
By virtue of the  embedding 
$W^{\widetilde{p},q}
\hookrightarrow W^{p,q }$, $\widetilde{p}\le p$,  
it is sufficient to show it 
in the case $1/p_1 + 1/p_2 = 1/p$.

Take 
$K$ and $\phi$ in the same way as in 
Proof of the latter inequality of 
Theorem \ref{th-amalgam-main}. 
In the present case, 
we take $\phi$ so that it satisfies 
the additional condition
\[
\sum_{m \in \Z^n}  \phi (\xi - m) 
\ge 1 
\;\; 
\text{for all}\;\; 
\xi \in \R^n.  
\]
Then we have 
\[
\|f\|_{W^{r,s} }
\approx 
\Big\|
\big\| 
\phi (D-\mu) f (x)
\big\|_{\ell^{s}_{\mu}} 
\Big\|_{L^r_x}
\]
for each 
$r,s\in (0, \infty]$.

We use the same representations 
as in 
Proof of the latter inequality of 
Theorem \ref{th-amalgam-main}:  
\begin{align*}
&
\Phi (\xi_1, \xi_2)
=
\sum_{k_1, k_2\in \Z^n}\, 
b(k_1, k_2) \, 
\phi_{k_1}\otimes \phi_{k_2} (\xi_1, \xi_2), 
\\
&
\sigma_{a, \Phi} (\xi_1, \xi_2)
=
\sum_{k_1, k_2\in \Z^n}\, 
b(k_1, k_2) \, 
\sigma_{ a, \phi_{k_1}\otimes \phi_{k_2}} 
(\xi_1, \xi_2)
\end{align*}
(see \eqref{eq-Phi-otimes} and 
\eqref{eq-TaPhi-Taotimes}).  
Recall that 
$\{ b(k_1, k_2)\} $ is a rapidly decreasing 
sequence. 
Hence in order to prove \eqref{eq-Wiener-amalgam-latter} 
it is sufficient to prove the estimate 
\begin{equation}\label{eq-Wiener-amalgam-latter-goal}
\big\| 
T_{a, \phi_{k_1}\otimes 
\phi_{k_2}}
\big\|_{W^{p_1, q_1}\times W^{p_2, q_2}\to W^{p, q}}
\le c 
\big\| 
S_a \big\|_{\ell^{q_1}\times \ell^{q_2}\to \ell^{q}} 
\end{equation}
(with $c$ 
independent of  
$k_1, k_2$).

Let 
$f_1, f_2 \in \calS (\R^n)$. 
We have 
\begin{align}
T_{  a, \phi_{k_1} \otimes \phi_{k_2} }
(f_1,f_2)(x)
&=
\sum_{\mu_1, \mu_2\in \Z^n}\,  
\iint_{\R^n\times \R^n} 
a(\mu_1, \mu_2) 
e^{2\pi i x \cdot (\xi_1 + \xi_2)} 
\nonumber 
\\
&
\quad \quad 
\times 
\phi_{k_1}(\xi_1 - \mu_1) 
\phi_{k_2}(\xi_2 - \mu_2) 
\widehat{f_1}(\xi_1)
\widehat{f_2}(\xi_2)\, 
d\xi_1 d\xi_2
\nonumber 
\\
&
=
\sum_{\mu_1, \mu_2\in \Z^n}\, 
a(\mu_1, \mu_2)\,  
g^{1}_{\mu_1, k_1}(x)
g^{2}_{\mu_2, k_2}(x), 
\label{Taotimesf1f2-F1F2}
\end{align}
where 
\begin{equation*}
g^{j}_{\mu_j, k_j}(x) = 
\phi_{k_j}(D- \mu_j) f_{j} (x)
=
e^{- 2\pi i K^{-1}k_j \cdot \mu_j} 
\phi (D- \mu_j) f_{j} (x + K^{-1}k_j), 
\quad j=1,2. 
\end{equation*}

Notice that 
$g^{j}_{\mu_j, k_j}$, $j=1,2$, satisfy 
\begin{align}
&
\supp \calF (g^{j}_{\mu_j, k_j})
\subset \{ \zeta \mid |\zeta - \mu_j| \lesssim 1\}, 
\label{FouriersuppFj}
\\
&
\Big\| 
\big\|g^{j}_{\mu_j, k_j}(x) \big\|_{\ell^{q_j}_{\mu_j}} 
\Big\|_{L^{p_j}_x}
=
\Big\| 
\big\|
\phi (D- \mu_j) f_{j} (x + K^{-1}k_j)
\big\|_{\ell^{q_j}_{\mu_j}} 
\Big\|_{L^{p_j}_x}
\approx 
\|f_j\|_{W^{p_j, q_j}}; 
\label{FjfjW}
\end{align}
notice that 
the quantities in 
\eqref{FjfjW} do not depend on $k_1, k_2$.

From 
\eqref{FouriersuppFj}, it follows that  
\begin{equation}\label{FouriersuppF1F2}
\supp \calF (g^{1}_{\mu_1, k_1}
g^{2}_{\mu_2, k_2})
\subset 
\{ \zeta \mid |\zeta - \mu_1 - \mu_2| \lesssim 1\} . 
\end{equation}
Let 
$\kappa$ be the function 
used in the definition of 
the quasi-norm of 
Wiener amalgam spaces. 
Then, since $\kappa$ has a compact support, 
we see that 
$\kappa (D - \mu) \big(
g^{1}_{\mu_1, k_1}
g^{2}_{\mu_2, k_2}\big) \neq 0$ 
only if $|\mu_1 + \mu_2 - \mu |\lesssim 1$. 
This fact and \eqref{Taotimesf1f2-F1F2} yield 
\begin{equation}\label{eq-kappaDmuh}
\kappa (D - \mu)
\big(  
T_{  a, \phi_{k_1} \otimes \phi_{k_2} }
(f_1,f_2)
\big) 
=\kappa (D-\mu) 
h_{\mu, k_1, k_2} 
\end{equation}
with 
\begin{equation*}
h_{\mu, k_1, k_2}
=
\sum_{|\tau|\lesssim 1}\; 
\sum_{
\mu_1 + \mu_2 = \mu + \tau}
\;  
a(\mu_1, \mu_2) 
g^{1}_{\mu_1, k_1} g^{2}_{\mu_2, k_2}. 
\end{equation*}
By 
\eqref{FouriersuppF1F2}, 
the Fourier transform of 
$h_{\mu, k_1, k_2 }$ has 
a compact support of diameter 
$\lesssim 1$. 
Hence \eqref{eq-kappaDmuh} and 
Lemma \ref{varphigmu} imply  
\begin{align*}
&
\big\|  
T_{  a, \phi_{k_1} \otimes \phi_{k_2} }
(f_1,f_2)
\big\|_{W^{p,q}}
=
\Big\|
\big\|  
\kappa (D - \mu) 
\big( T_{  a, \phi_{k_1} \otimes \phi_{k_2} }
(f_1,f_2)
\big) (x) 
\big\|_{\ell^{q}_{\mu}} 
\Big\|_{L^p_x}
\\
&
=
\Big\|
\big\|
\kappa (D - \mu) h_{\mu, k_1, k_2}(x) \big\|_{\ell^q_{\mu}} 
\Big\|_{L^p_x}
\lesssim 
\Big\|
\big\|
h_{\mu, k_1, k_2}(x) 
\big\|_{\ell^q_{\mu}} 
\Big\|_{L^p_x}. 
\end{align*}
Using the definition of 
$\|S_a\|_{\ell^{q_1}\times \ell^{q_2}\to \ell^q}$  
and H\"older's inequality with 
exponents $1/p_1 + 1/p_2 = 1/p$, 
we obtain 
\begin{align*}
&
\Big\|
\big\|
h_{\mu, k_1, k_2}(x) 
\big\|_{\ell^q_{\mu}} 
\Big\|_{L^p_x}
\\
&
=
\Bigg\|
\bigg\| 
\sum_{|\tau|\lesssim 1}\, 
\sum_{\mu_1 + \mu_2 = \mu + \tau} 
a(\mu_1, \mu_2)
g^{1}_{\mu_1, k_1}(x) g^{2}_{\mu_2, k_2} (x)
\bigg\|_{\ell^q_{\mu}} 
\Bigg\|_{L^p_{x}}
\\
&
\lesssim 
\Bigg\|
\bigg\| 
\sum_{
\mu_1 + \mu_2 = \mu} 
a(\mu_1, \mu_2)
g^{1}_{\mu_1, k_1}(x) g^{2}_{\mu_2, k_2}(x)
\bigg\|_{\ell^q_{\mu}} 
\Bigg\|_{L^p_{x}}
\\
&
\le 
\left\|
\|S_a\|_{\ell^{q_1}\times \ell^{q_2}\to \ell^q}
\|g^{1}_{\mu_1, k_1}(x)\|_{\ell^{q_1}_{\mu_1}}
\|g^{2}_{\mu_2, k_2}(x)\|_{\ell^{q_2}_{\mu_2}}
\right\|_{L^p_x}
\\
&
\le 
\|S_a\|_{\ell^{q_1}\times \ell^{q_2}\to \ell^q}
\Big\| 
\big\| 
g^{1}_{\mu_1, k_1}(x)
\big\|_{\ell^{q_1}_{\mu_1}} 
\Big\|_{L^{p_1}_x}
\Big\|
\big\| 
g^{2}_{\mu_2, k_2}(x)
\big\|_{\ell^{q_2}_{\mu_2}} 
\Big\|_{ L^{p_2}_x}.  
\end{align*}

Now 
combing the above inequalities with 
\eqref{FjfjW}, 
we obtain 
\[
\big\|  
T_{  a, \phi_{k_1} \otimes \phi_{k_2} }
(f_1,f_2)
\big\|_{W^{p,q}}
\lesssim 
\|S_a\|_{\ell^{q_1}\times \ell^{q_2}\to \ell^q}
\|f_1\|_{W^{p_1, q_1}}
\|f_2\|_{W^{p_2, q_2}}, 
\]
which implies 
\eqref{eq-Wiener-amalgam-latter-goal}. 
Thus 
the latter 
inequality of Theorem \ref{th-Wiener-amalgam-main} 
is proved. 
\end{proof}

Next, we shall 
prove 
the former inequality of 
Theorem \ref{th-Wiener-amalgam-main}.

\begin{proof}[{Proof of the former inequality of 
Theorem \ref{th-Wiener-amalgam-main}}] 
Here we shall prove 
the inequality 
\begin{equation}\label{eq-Wiener-amalgam-former}
\|S_a\|_{\ell^{q_1}\times \ell^{q_2}\to \ell^q}
\le c 
\left\|
T_{a, \Phi }
\right\|_{W^{p_1, q_1}\times W^{p_2, q_2} \to W^{p,q}} 
\end{equation}

Since $\Phi $ satisfies the condition (B), by the same reason as in 
Proof of the former inequality 
of Theorem \ref{th-amalgam-main}, 
we can take 
$\boldsymbol{\xi}^{0}=(\xi^0_1, \xi^0_2)\in \R^{2n}$ 
and functions 
$\theta_1, \theta_2\in C_{0}^{\infty}(\R^n)$ that satisfy 
\eqref{Phineq0}, 
\eqref{Phisupp}, 
\eqref{thetajsupp}, 
and \eqref{thetaPhi}.

We take a function  
$\kappa\in C_{0}^{\infty}(\R^n)$ 
such that 
\begin{equation*}
\sum_{\mu \in \Z^n} \kappa (\xi - \mu)= 1
\;\;\text{for all}\;\; 
\xi \in \R^n
\end{equation*}
and 
\begin{equation}\label{kappa2}
|\xi|<2 \epsilon \; 
\Rightarrow 
\; 
\kappa (\xi -\mu) =
\begin{cases}
{1} & \text{if $\mu=0$, } \\
{0} & \text{otherwise.}
\end{cases}
\end{equation}
where 
$\epsilon$ is the number in \eqref{Phisupp}. 
Such a $\kappa$ certainly exists  
if $\epsilon$ is chosen sufficiently small.

Now let 
$b_1, b_2\in X(\Z^n)$. 
We define 
$f_1, f_2\in \calS (\R^n)$ through Fourier transform  
by 
\[
\widehat{f_j} (\xi) = \sum_{\nu \in \Z^n}
b_j (\nu) \theta_j (\xi - \nu ), 
\quad 
\xi \in \R^n. 
\]

From 
\eqref{thetajsupp} and 
\eqref{kappa2}, we have 
\begin{equation*}
\kappa (\xi - \xi^0_j - \mu)
\widehat{f_j} (\xi) 
=
\kappa (\xi - \xi^0_j - \mu)
\sum_{\nu \in \Z^n}
b_j (\nu)
\theta_{j} (\xi - \nu)
=
b_j (\mu)
\theta_{j} (\xi - \mu)
\end{equation*}
and hence 
\[
\kappa (D - \xi^0_j - \mu)
f_j (x) 
=
b_j (\mu)
e^{2\pi i \mu \cdot x}
\calF^{-1}\theta_{j} (x). 
\]
Thus 
\begin{equation}\label{fjWFjqj}
\begin{split}
&
\|f_j\|_{W^{p_j, q_j}}
\approx 
\Big\|
\big\| 
\kappa (D - \xi^0_j - \mu)
f_j (x) 
\big\|_{\ell^{q_j}_{\mu}} 
\Big\|_{L^{p_j}_x}
\\
&
=
\Big\|
\big\| 
b_j (\mu)
e^{2\pi i \mu \cdot x}
\calF^{-1}\theta_{j} (x)
\big\|_{\ell^{q_j}_{\mu}} 
\Big\|_{L^{p_j}_x}
=
\|b_j\|_{\ell^{q_j}}
\|
\calF^{-1}\theta_j (x) 
\|_{L^{p_j}_x}
=c 
\|b_j\|_{\ell^{q_j}}. 
\end{split}
\end{equation}

On the other hand, 
just in the same way as we obtained 
\eqref{eq-iii} 
in Proof of the former inequality of 
Theorem \ref{th-amalgam-main}, 
we obtain 
\begin{align}
T_{a, \Phi}(f_1,f_2)(x)
&=
\sum_{ \mu_1, \mu_2 \in \Z^n }
\sum_{ \nu_1, \nu_2 \in \Z^n }
a(\mu_1, \mu_2)
b_1(\nu_{1})
b_2(\nu_{2})\, 
\iint_{\R^n \times \R^n}
e^{2\pi i x \cdot (\xi_1 + \xi_2 )} 
\nonumber 
\\
&
\quad \quad \times 
\Phi (\xi_1- \mu_1, \xi_2 -\mu_2)
\theta_1 (\xi_1 - \nu_1)
\theta_2 (\xi_2 - \nu_2)
\, d\xi_1 d\xi_2
\nonumber 
\\
&=
\sum_{ \mu_1, \mu_2 \in \Z^n }
a(\mu_1, \mu_2)
b_1(\mu_{1})
b_2(\mu_{2})\, 
e^{2\pi i x\cdot (\mu_1 + \mu_2)} 
g(x) 
\label{eq-Taotimesf1f2-F1F2modg}
\end{align}
with $g(x)$ given by 
\eqref{eq-def-g}. 
Observe that 
\begin{align*}
\supp \calF \big( e^{2\pi i x\cdot (\mu_1 + \mu_2)} 
g(x)\big) 
&
\subset 
\big\{ 
\xi_1 + \xi_2 + \mu_1 + \mu_2 
\mid 
\xi_1 \in 
\supp \theta_1, 
\;\; 
\xi_2 \in 
\supp \theta_2, 
\big\}
\\
&
\subset 
\big\{ 
\xi^{0}_1 + \xi^{0}_2 +\zeta + \mu_1 + \mu_2 
\mid 
\zeta \in 
\R^n 
\;\; 
|\zeta|< 2\epsilon 
\big\}. 
\end{align*}
Hence our choice of $\kappa$ (see \eqref{kappa2})  
implies that 
\begin{align*}
& 
\mu =  \mu_1 + \mu_2\;\; 
\Rightarrow 
\;\; 
\kappa (\xi - \xi^{0}_1 - \xi^{0}_2 - \mu)=1 
\;\; \text{on}\;\; 
\supp \calF \big( e^{2\pi i x\cdot (\mu_1 + \mu_2)} 
g(x)\big), 
\\
&
\mu \neq  \mu_1 + \mu_2\;\; 
\Rightarrow 
\;\; 
\kappa (\xi - \xi^{0}_1 - \xi^{0}_2 - \mu)=0 
\;\; \text{on}\;\; 
\supp \calF \big( e^{2\pi i x\cdot (\mu_1 + \mu_2)} 
g(x)\big), 
\end{align*}
and hence 
\begin{equation*}
\kappa (D_x - \xi^0_1 - \xi^0_2 -\mu)
\big( e^{2\pi i x\cdot (\mu_1 + \mu_2)} 
g(x)\big) 
=
\begin{cases}
{e^{2\pi i x\cdot (\mu_1 + \mu_2)} 
g(x)} 
& \text{if $\mu_1 + \mu_2 = \mu$, } \\
{0} & \text{otherwise.}
\end{cases}
\end{equation*}
This relation and 
\eqref{eq-Taotimesf1f2-F1F2modg} imply 
\begin{align*}
&
\kappa (D_x - \xi^0_1 - \xi^0_2 -\mu) 
T_{a, \Phi}
(f_1,f_2)(x)
\\
&
=\sum_{ \mu_1+\mu_2=\mu  }\, 
a(\mu_1, \mu_2)
b_1(\mu_{1})
b_2(\mu_{2})\, 
e^{2\pi i x\cdot \mu } g(x) . 
\end{align*}
Recall that $|g(x)|\ge 1$ on $Q$ 
(see \eqref{thetaPhi}). 
Hence 
\begin{align*}
&
\|T_{a, \Phi}(f_1,f_2)\|_{W^{p,q}}
\\
&
\approx 
\Big\|
\big\|
\kappa (D_x - \xi^0_1 - \xi^0_2 -\mu)
T_{a, \Phi} (f_1,f_2)(x)
\big\|_{\ell^q_{\mu}} 
\Big\|_{L^p_x}
\\
&
=
\Bigg\|
\bigg\| 
\sum_{\mu_1+\mu_2=\mu }
a(\mu_1, \mu_2)
b_1(\mu_{1})
b_2(\mu_{2})\, 
e^{2\pi i x \cdot \mu} 
g(x) 
\bigg\|_{\ell^q_{\mu}} 
\Bigg\|_{L^p_x}
\\
&
=
\bigg\|
\sum_{\mu_1+\mu_2=\mu }
a(\mu_1, \mu_2)
b_1(\mu_{1})
b_2(\mu_{2})
\bigg\|_{\ell^{q}_{\mu}}
\, 
\| g \|_{L^p}
\\
&
\ge  
\bigg\|
\sum_{ \mu_1+\mu_2=\mu }
a(\mu_1, \mu_2)
b_1(\mu_{1})
b_2(\mu_{2})\bigg\|_{\ell^{q}_{\mu}}
=
\|S_a (b_1, b_2)\|_{\ell^q}. 
\end{align*}

Combining the above inequalities  
with \eqref{fjWFjqj}, 
we obtain 
\begin{align*}
&
\|S_a (b_1, b_2)\|_{\ell^q}
\lesssim 
\| 
T_{a, \Phi} (f_1, f_2)\|_{W^{p,q}}
\\
&
\le 
\|T_{a, \Phi}\|_{W^{p_1,q_1} \times W^{p_2, q_2} \to W^{p,q}}
\|f_1\|_{W^{p_1, q_1}} 
\|f_2\|_{W^{p_2, q_2}}
\\
&
\lesssim 
\|T_{a,\Phi}\|_{W^{p_1,q_1} \times W^{p_2, q_2} \to W^{p,q} }
\|b_1\|_{\ell^{q_1}} 
\|b_2\|_{\ell^{q_2}}, 
\end{align*}
which implies 
\eqref{eq-Wiener-amalgam-former}. 
Now the former inequality 
of 
Theorem \ref{th-Wiener-amalgam-main} 
is proved 
and hence 
the proof 
of Theorem \ref{th-Wiener-amalgam-main} 
is complete. 
\end{proof}

\section{Appendix}\label{sectionappendix}

Here we give proofs of the facts mentioned in 
Remark \ref{remark-amalgam} (3) and 
Remark \ref{remark-Wiener-amalgam} (2).

\begin{lmmx}\label{lem-Holder-A}
Let 
$\sigma \in L^{\infty}(\R^n \times \R^n) $, 
$\sigma \neq 0$, 
$p_1, p_2, p, q_1, q_2, q \in (0, \infty]$, and 
suppose  
$T_{\sigma}$ is bounded in 
$(L^{p_1}, \ell^{q_1}) \times  (L^{p_2}, \ell^{q_2}) \to (L^p, \ell^q)$. 
Then $1/q \le 1/q_1 + 1/q_2$. 
\end{lmmx}

\begin{proof}
Take a function $\varphi \in \calS (\R^n)$ such that 
$\supp \widehat{\varphi} \subset \{|\xi|\le 1\}$ 
and 
$|\varphi (x)|\ge 1$ for $x \in Q$. 
Take a Lebesgue point $(\xi_0, \eta_0)$ of 
$\sigma$ such that 
$\sigma (\xi_0, \eta_0)\neq 0$ and  
define 
$f_{\epsilon}$ and $g_{\epsilon}$ for $0<\epsilon <1$ 
by 
\begin{align*}
&
\widehat{f_{\epsilon}}(\xi) 
=\epsilon^{-n} \widehat{\varphi} \big( \epsilon^{-1}(\xi -\xi_0)\big) , 
\quad 
\widehat{g_{\epsilon}}(\eta) 
= \epsilon^{-n} \widehat{\varphi} \big( \epsilon^{-1}(\eta -\eta_0)\big), 
\\
&
f_{\epsilon}(x) = e^{2\pi i \xi_0 \cdot x} \varphi (\epsilon x), 
\quad 
g_{\epsilon}(x) = e^{2\pi i \eta_0 \cdot x} \varphi (\epsilon x). 
\end{align*}

Then 
$\ichi_{Q}(\epsilon x) \le 
|f_{\epsilon}(x)| 
=|g_{\epsilon}(x)| 
\lesssim (1+ \epsilon |x|)^{-N}$ with any $N>0$. 
From this we easily see that 
$\big\|f_\epsilon\big\|_{(L^{p_1}, \ell^{q_1})}
\approx  
\epsilon^{-n/q_1}$ 
and 
$\big\|g_\epsilon\big\|_{(L^{p_2}, \ell^{q_2})}
\approx  
\epsilon^{-n/q_2}$ 
for $0<\epsilon < 1$.

On the other hand, $T_{\sigma} (f_{\epsilon}, g_{\epsilon})(x)$ is written as 
\begin{align*}
T_{\sigma} (f_{\epsilon}, g_{\epsilon})(x)
&
=
\iint 
e^{2\pi i x \cdot (\xi+\eta)} 
\big( \sigma(\xi, \eta) - \sigma (\xi_0, \eta_0) \big) 
\widehat{f_{\epsilon}}(\xi)
\widehat{g_{\epsilon}}(\eta)
\, d\xi d\eta
\\
&
\quad 
+\iint 
e^{2\pi i x \cdot (\xi+\eta)} 
\sigma(\xi_0, \eta_0) 
\widehat{f_{\epsilon}}(\xi)
\widehat{g_{\epsilon}}(\eta)
\, d\xi d\eta
=
A + B, 
\quad \text{say}. 
\end{align*}
Since 
$(\xi_0, \eta_0)$ is a Lebesgue point of 
$\sigma$, the term $A$ tends to $0$ uniformly in 
$x \in \R^n$ as $\epsilon \to 0$. 
For the term $B$, 
we have 
$B
=\sigma(\xi_0, \eta_0) 
e^{2\pi i x \cdot (\xi_0 +\eta_0)} 
\varphi (\epsilon x)^2$, 
and hence our choice of $\varphi$ implies 
$|B|\ge |\sigma (\xi_0, \eta_0)|\ichi_{Q}(\epsilon x)$. 
Hence for all sufficiently small $\epsilon$ we have 
$|T_{\sigma}(f_{\epsilon}, g_{\epsilon})(x)| 
\ge 
2^{-1}|\sigma (\xi_0, \eta_0)|\, \ichi_{Q}(\epsilon x)$   
and thus 
\begin{align*}
&
\big\| T_{\sigma}(f_{\epsilon}, g_{\epsilon}) \big\|_{(L^p, \ell^q)}
=
\Big\| 
\big\|
 T_{\sigma}(f_{\epsilon}, g_{\epsilon}) (z+ \rho)
\big\|_{L^p_z (Q)}
\Big\|_{\ell^{q}_{\rho}(\Z^n)}
\\
&
\gtrsim 
\Big\| 
\big\| 
 |\sigma (\xi_0, \eta_0)|\, 
 \ichi_{Q} \big(\epsilon (z + \rho) \big)
 \big\|_{L^p_z (Q)}
\Big\|_{\ell^{q}_{\rho}(\Z^n)}
\approx 
 |\sigma (\xi_0, \eta_0)|\, 
\epsilon ^{-n/q}. 
\end{align*}

If $T_{\sigma}$ is bounded 
in 
$(L^{p_1}, \ell^{q_1}) \times  (L^{p_2}, \ell^{q_2}) \to (L^p, \ell^q)$, 
then the inequalities obtained above imply 
$\epsilon ^{-n/q}
=O\big( 
\epsilon^{-n/q_1} \epsilon^{-n/q_2}\big)$ 
as $\epsilon \to 0$, 
which holds only when $1/q \le 1/q_1 + 1/q_2$. 
\end{proof}

\begin{lmmx}\label{lem-Holder-WA}
Let 
$\sigma \in L^{\infty}(\R^n \times \R^n) $, 
$\sigma \neq 0$,  
$p_1, p_2, p, q_1, q_2, q \in (0, \infty]$, and suppose 
$T_{\sigma}$ is 
bounded in 
$W^{p_1, q_1} \times  W^{p_2, q_2} \to W^{p, q}$. 
Then 
$1/p \le 1/p_1 + 1/p_2$. 
\end{lmmx}

\begin{proof}
Take $\varphi$, 
$(\xi_0, \eta_0)$, $f_{\epsilon}$, and $g_{\epsilon}$ 
in the same way as in 
Proof of Lemma \ref{lem-Holder-A}. 

%
To estimate 
the quasi-norms of 
$f_{\epsilon}$ and $g_{\epsilon}$ in Wiener-amalgam spaces, 
take a function 
$\kappa \in C_{0}^{\infty}(\R^n)$ such that 
$\sum_{\mu \in \Z^n} \kappa (\xi - \mu) = 1$  
for all $\xi \in \R^n$ 
and that 
\[
|\xi |<1/10
\;\;
\Rightarrow \;\; 
\kappa (\xi )=1\;\; 
\text{and}\;\; 
\kappa (\xi - \mu)=0\;\; 
\text{for}\;\; 
\mu \in \Z^n \setminus \{0\}.  
\]
Then, for $0<\epsilon < 1/10$, 
we have 
$\kappa (D-\xi_0) 
f_{\epsilon}= f_{\epsilon}$ and 
$\kappa (D - \xi_0 - \mu) f_{\epsilon} = 0$ 
for $\mu \in \Z^n \setminus \{0\}$, 
and thus 
\begin{align*}
&
\big\|f_\epsilon\big\|_{W^{p_1, q_1}}
\approx 
\Big\| 
\big\| 
\kappa (D-\xi_0 - \mu)
f_\epsilon (x)
\big\|_{\ell^{q_1}_{\mu}(\Z^n)}
\Big\|_{L^{p_1}_{x}(\R^n)}
\\
&
=
\Big\| 
f_\epsilon (x)
\Big\|_{L^{p_1}_{x}(\R^n)}
=
\Big\| \varphi (\epsilon x) 
\Big\|_{L^{p_1}_{x}(\R^n)}
\approx 
\epsilon^{-n/p_1}.  
\end{align*}
Similarly 
we have 
$\big\|g_\epsilon\big\|_{W^ {p_2, q_2}}
\approx \epsilon^{-n/p_2}$ for 
$0<\epsilon <1/10$.

To estimate the $W^{p,q}$-quasi-norm of 
$T_{\sigma}(f_{\epsilon}, g_{\epsilon})$, we 
take a function 
$\widetilde{\kappa} \in C_{0}^{\infty}(\R^n)$ such that 
$
\widetilde{\kappa} (\xi_0 + \eta_0 )\neq 0$ 
and 
$\sum_{\mu \in \Z^n} \widetilde{\kappa} (\xi - \mu) \ge 1$ 
for all $\xi \in \R^n$.  
Then 
\begin{align*}
&
\widetilde{\kappa}(D) 
T_{\sigma} (f_{\epsilon}, g_{\epsilon})(x)
=
\iint 
e^{2\pi i x \cdot (\xi+\eta)} 
\widetilde{\kappa}(\xi + \eta)
\sigma(\xi, \eta) 
\widehat{f_{\epsilon}}(\xi )
\widehat{g_{\epsilon}}(\eta )
\, d\xi d\eta
\\
&
=
\iint 
e^{2\pi i x \cdot (\xi+\eta)} 
\bigg[ 
\widetilde{\kappa}(\xi + \eta)
\sigma(\xi, \eta) - 
\widetilde{\kappa}(\xi_0 + \eta_0)
\sigma (\xi_0, \eta_0) \bigg] 
\widehat{f_{\epsilon}}(\xi )
\widehat{g_{\epsilon}}(\eta )
\, d\xi d\eta
\\
&
\quad 
+\iint 
e^{2\pi i x \cdot (\xi+\eta)} 
\widetilde{\kappa}(\xi_0 + \eta_0)
\sigma (\xi_0, \eta_0)
\widehat{f_{\epsilon}}(\xi )
\widehat{g_{\epsilon}}(\eta )
\, d\xi d\eta
=
A+ B, 
\quad \text{say}.  
\end{align*}
Since 
$(\xi_0, \eta_0)$ is a Lebesgue point 
of 
$\widetilde{\kappa}(\xi + \eta)
\sigma (\xi, \eta) $, 
the term $A$ tends to $0$ 
uniformly in $x\in \R^n$ 
 as $\epsilon \to 0$.   
For the term $B$, we have 
$B
=
\widetilde{\kappa}(\xi_0 + \eta_0)
\sigma(\xi_0, \eta_0) 
e^{2\pi i x \cdot (\xi_0 + \eta_0)} 
\varphi (\epsilon x)^2  
$
and hence 
$|B|\ge |\widetilde{\kappa}(\xi_0 + \eta_0)\sigma (\xi_0, \eta_0)| \, 
\ichi _{Q}(\epsilon x)$. 
Hence for all sufficiently small $\epsilon $ we have 
$
\big|\widetilde{\kappa}(D)T_{\sigma}(f_{\epsilon}, g_{\epsilon})(x)
\big| \ge 
2^{-1}\,  
|\widetilde{\kappa}(\xi_0 + \eta_0)\sigma (\xi_0, \eta_0)| \, 
\ichi_{Q}(\epsilon x)$ 
and thus 
\begin{align*}
&
\big\| T_{\sigma}(f_{\epsilon}, g_{\epsilon}) \big\|_{W^{p, q}}
\approx 
\Big\| 
\big\|
\widetilde{\kappa} (D-\mu) 
T_{\sigma}(f_{\epsilon}, g_{\epsilon}) (x)
\big\|_{\ell^{q}_{\mu}(\Z^n)}
\Big\|_{L^p_x (\R^n)}
\\
&
\ge 
\big\| 
\widetilde{\kappa} (D) 
T_{\sigma}(f_{\epsilon}, g_{\epsilon}) (x)
\big\|_{L^p_x (\R^n)}
\gtrsim  
\Big\| 
|\widetilde{\kappa}(\xi_0 + \eta_0)\sigma (\xi_0, \eta_0)|\, 
\ichi_{Q} (\epsilon x) 
\Big\|_{L^{p}_{x}(\R^n)}
\\
&
\approx 
|\widetilde{\kappa}(\xi_0 + \eta_0)\sigma (\xi_0, \eta_0)|\,
\epsilon ^{-n/p}. 
\end{align*}

If  
$T_{\sigma}$ 
is bounded in 
$W^{p_1, q_1} \times  W^{p_2, q_2} \to W^{p, q}$, 
then 
the inequalities obtained above 
imply 
$\epsilon ^{-n/p}
=O\big( 
\epsilon^{-n/p_1} \epsilon^{-n/p_2}\big)$ as 
$\epsilon \to 0$, 
which holds only if $1/p \le 1/p_1 + 1/p_2$. 
\end{proof}


\end{document}